\documentclass[notitlepage,12pt]{article}
\usepackage[utf8]{inputenc}
\usepackage{amssymb}
\usepackage{amsfonts}
\usepackage{amsmath}
\usepackage{enumitem}

\newtheorem{theorem}{Theorem}

\newtheorem{example}[theorem]{Example}

\newtheorem{lemma}[theorem]{Lemma}

\newtheorem{proposition}[theorem]{Proposition}
\newtheorem{remark}[theorem]{Remark}

\newenvironment{proof}[1][Proof]{\noindent\textbf{#1.} }{\ \rule{0.5em}{0.5em}}

\newcommand{\setR}{\mathbb{R}}
\newcommand{\setN}{\mathbb{N}}
\newcommand{\dualspace}[1]{{#1}^{\ast}}
\newcommand{\then}{\Rightarrow}

\newcommand{\Ck}[2][]{\operatorname{C}^{{#1}} \left( {#2} \right)}

\renewcommand{\lim}{\operatornamewithlimits{lim}\limits}
\renewcommand{\liminf}{\operatornamewithlimits{liminf}\limits}

\renewcommand{\inf}{\operatornamewithlimits{inf}\limits}

\renewcommand{\min}{\operatornamewithlimits{min}\limits}
\renewcommand{\max}{\operatornamewithlimits{max}\limits}
\newcommand{\summ}[2][]{\sum\limits^{#1}_{#2} }
\newcommand{\intt}[2][]{\int\limits_{{#2}}^{{#1}}}

\newcommand{\foralll}{\operatornamewithlimits{\forall}\limits}
\newcommand{\existss}{\operatornamewithlimits{\exists}\limits}

\newcommand{\modul}[2][]{\left| {#2} \right|_{#1} }
\newcommand{\norm}[2][]{\left\| {#2} \right\|_{#1}}
\newcommand{\dual}[3][]{\left\langle {#2},{#3} \right\rangle_{#1}}
\newcommand{\set}[1]{\left\lbrace {#1}\right\rbrace}
\newcommand{\seq}[3][\setN]{\left( {#2}_{#3} \right)_{{#3} \in {#1}} }
\newcommand{\boundary}[1]{\partial {#1}}
\newcommand{\signum}{\operatorname{sgn}}

\newcommand{\f}[2][]{{#1} \left({#2}\right)}
\newcommand{\deltaop}[2][]{\Delta {#1}\left({#2}\right)}

\begin{document}

\author{Marek Galewski, Piotr Kowalski}
\title{Three solutions to discrete anisotropic problems with two parameters}
\date{\today}
\maketitle

\begin{abstract}
In this note we derive a~type of a~three critical point theorem which we
further apply to investigate the multiplicity of solutions to discrete
anisotropic problems with two parameters. 
\end{abstract}

\section{Introduction}

The main aim of this note is to develop further a~type of
the three critical point theorem by providing some general version which
would be applicable for various types of nonlinear problems depending on
numerical parameters. 
The main result of this note says that a coercive functional acting on a reflexive strictly convex Banach space under some geometric conditions concerning local behaviour around $0$ has at least three critical points.
The research connected with the existence of at least
three critical points to action functionals, both smooth and nonsmooth,
connected with boundary value problems has received some considerable
attention lately. It begun with the celebrated results of Ricerri \cite
{ricceri.2,ricceri.3} and was further developed in many subsequent
papers, see for example \cite{ricceri.4,ricceri.5}. The three critical
theorem was later generalized, simplified and next extensively applied, see
for example \cite{molica.bisci.bonanno,bonnano.chinne} and references in \cite{ricceri.1}. Recently another type of a~three critical point theorem was developed in 
\cite{cabada.iannizzotto.tersian} and further generalized in \cite{cabada.iannizzotto} to the case of $p-$%
laplacians and in \cite{galewski.wieteska} to the case of anisotropic problems. In
this note, we base ourselves on results in \cite{cabada.iannizzotto,galewski.wieteska} in order to provide yet another type of a~three critical point
theorem, which would hold for problems to which the results mentioned cannot
be applied. Moreover, our main result generalizes main theorems in \cite{cabada.iannizzotto,galewski.wieteska}. 

As a~model problem to which our general multiplicity results could be
applied is the following discrete boundary value problem:
\begin{equation}
\left\{ 
\begin{array}{l}
\begin{array}{l}
-\deltaop{\modul{\deltaop[x]{k-1}}^{p(k-1)-2} \deltaop[x]{k-1}} +\\ 
+\gamma \f[g]{k,x(k)} +\lambda \f[f]{k,x(k)}=0
\end{array}
,\: k\in \left[ 1,T\right] \\ 
x(0)=x(T+1)=0,
\end{array}
\right.  \label{zad}
\end{equation}
where $\gamma ,\lambda >0$ are numerical parameters, $f,g:[1,T]\times 
\setR \rightarrow \setR$ are continuous functions subject to some
assumptions, $[1,T]$ is a~discrete interval $\{1,2,...,T\}$, $\deltaop[x]{k-1}=x(k)-x(k-1)$ is the forward difference operator, $p: \left[ 0,T+1\right] \rightarrow \setR_{+}$, $p^{-}=\min_{k\in \left[ 0,T+1\right]} p(k) >1,$ $p^{+}=\max_{k\in \left[ 0,T+1\right] }p(k) $. Solutions to \eqref{zad} will be investigated in a~space 
\begin{equation*}
X=\set{x:[0,T+1]\rightarrow \setR :x(0)=x(T+1)=0}
\end{equation*}
which considered with a~norm $\norm{x}=\f{ \summ[T+1]{k=1}\modul{\deltaop[x]{k-1}}^2}^\frac{1}{2}$ becomes a~Hilbert space. 

The research concerning the discrete anisotropic problems of type \eqref{zad} have only been started, \cite{kone.ouaro}, \cite{mihailescu.radulescu.tersian}, where known tools from the critical point theory are applied in order to get the existence of solutions. 
In \cite{bereanu.jebelean.serban} the authors undertake the
existence of periodic or Neumann solutions for the discrete $p(k)-$%
Laplacian. The so called ground state solutions are considered in \cite{bereanu.jebelean.serban.2}.
Continuous version of problems like \eqref{zad} are known to be mathematical models of various phenomena arising in the study of elastic mechanics, \cite{zhikov}, electrorheological fluids, \cite{ruzicka}, or image restoration, \cite{chen.levine.rao}. Variational continuous anisotropic problems have been started by Fan and Zhang in \cite{fan.zhang} and later considered by many methods and authors, \cite{harjulehto.hasto.le.nuortio}, for an extensive survey of such boundary value problems. 

For some related papers let us also mention, far from being exhaustive, the following \cite{agarwal.perera.oregan,cai.yu,liu.su,yang.zhang,cheng.zhang,zhang}. These papers employ in the discrete setting the variational techniques already known for continuous problems, of course with necessary modifications. The tools employed cover the Morse theory, the mountain pass methodology and linking arguments.

Paper is organized as follows. Firstly we provide a~variational framework
and assumptions for problem \eqref{zad} in Section 2. Next, in Section 3 we comment
on three critical point theorems which we apply. In Section 4 we give a
general multiplicity result which we apply for problem \eqref{zad} in Section 5.

\section{Variational framework}

In this section we provide a~variational framework for problem \eqref{zad}. We connect solutions to \eqref{zad} with critical points to the following action functional 
\begin{equation*}
E_{\gamma ,\lambda }(x)=\summ[T+1]{k=1} \frac{1}{p(k-1)}\modul{\deltaop[x]{k-1}}^{p(k-1)} +\lambda \summ[T]{k=1} F(k,x(k))+\gamma \summ[T]{k=1}G(k,x(k)),
\end{equation*}%
where $F(k,s)=\intt[s]{0}f(k,t)dt$, $G(k,s)=\intt[s]{0}g(k,t)dt $. With any fixed $\gamma ,\lambda >0$ functional $E_{\gamma ,\lambda }$ is differentiable in the sense of G\^{a}teaux. Its G\^{a}teaux derivative reads
\begin{equation*}
\begin{array}{l}
\dual{E_{\gamma ,\lambda }^{^{\prime }}(x)}{v}=\summ[T+1]{k=1}\modul{\deltaop[x]{k-1}}^{p(k-1)-2} \deltaop[x]{k-1} \deltaop[v]{k-1} +\\
 +\lambda \summ[T]{k=1} f(k,x(k)) v(k) +\gamma \summ[T]{k=1}g(k,x(k))v(k).
\end{array}
\end{equation*}
A critical point to $E_{\gamma ,\lambda }$ is a~point $x\in X$ such that $\dual{E_{\gamma ,\lambda }^{^{\prime }}(x)}{v}=0$ for all 
$v\in X$ and is a~weak solution to \eqref{zad}. Summing by parts we see that
any weak solution to \eqref{zad} is in fact a~strong one. Hence in order to
solve \eqref{zad} we need to find critical points to $E_{\gamma ,\lambda }$
and further investigate their multiplicity. We will need the following
assumptions.

\textit{\begin{enumerate}[label=\textbf{(A.\arabic*)}]
\item $f:[1,T]\times \setR \rightarrow \setR$ is a
continuous function such that
\begin{equation*}
\liminf_{|t|\rightarrow \infty }\frac{F(k,t)}{|t|^{p^{-}}}\geq 0\text{ for
any }k\in [1,T]
\end{equation*}\label{A.1}
\item There exist numbers $m>0$, $s_{2}\geq s_{1}>m$ such that $\f[F]{k,t} >0$ for $(k,t) \in [1,T]\times \f{\left[ -m,m\right] \backslash \set{0}} $ and $\f[F]{k,t} <0$ for $(k,t) \in [1,T]\times \left[ s_{1},s_{2}\right] $. \label{A.2}
\item there exists $M_{1} >0$ such that \label{A.4}
\begin{equation*}
\begin{array}{c}
G(k,t) \leq 0 \: \text{ for all } (k,t) \in [1,T] \times \left[-M_{1},M_{1}\right]\\\\
\liminf_{\modul{t}\to +\infty} \frac{G(k,t)}{\modul{t}}>0 \: \text{ for all } k \in [1,T]
\end{array}
\end{equation*}
\item For all $k \in [1,t]$ and $x \in [-m,m]$ function \label{A.5}
\begin{equation*}
x \rightarrow f(k,x)
\end{equation*} is non-decreasing.
\end{enumerate}}

We note that the assumptions on $f$ are similar to those considered in \cite{galewski.wieteska} but this problem cannot be easily tackled by method from \cite{galewski.wieteska} since we have another term $g$ which also depends on a~numerical parameter. That is why we must provide another three critical point theorem in order to investigate the multiplicity of solutions. 

Now we provide example of nonlinear terms which satisfy assumptions \ref{A.1}-\ref{A.5}.

\begin{example}
Let $T$ be a~positive integer, $T\geq 2.$ Let us consider a~continuous
function $f:\left[ 1,T\right] \times \mathbb{R}\rightarrow \mathbb{R}$ given
by the formula

\begin{equation*}
f(k,x)=\left\{ 
\begin{array}{l l}
\alpha(k) \cdot \frac{1}{2} x&,|x|<2\\
\alpha(k) \cdot  \f{-x+3 \signum (x)}&,2\leq |x|<4\\
\alpha(k) \cdot \f{-\signum (x)}&,4\leq |x|<6\\
\alpha(k) \cdot \f{x-7 \signum (x)}&,6\leq |x|<8\\
\alpha(k) \cdot \signum(x) e^{-|x|+8}&, |x|\geq 8
\end{array}
\right. 
\end{equation*}
where $\alpha [1,T]\rightarrow (0,+\infty)$ is an arbitrary function. Let us consider another function
$g:\left[ 1,T\right] \times \setR \rightarrow \setR$ given
by the formula 
\begin{equation*}
g(k,x) = \beta(k) \cdot \f{0.5- e^{-x^2}}
\end{equation*}
where $\beta [1,T]\rightarrow (0,+\infty)$ is an arbitrary function. Then \ref{A.1} is satisfied since $\lim_{x->\infty} F(k,x) = \intt[\infty]{0} f(k,x) = 0$ for every $k \in [1,T]$. \ref{A.2} and \ref{A.5} are also satisfied with $m=2$ and $s_2=s_1 = 6$. Then for any $x \in [-2,2] \setminus \set{0}$ and $k \in [1,T]$
\begin{equation*}
F(k,x) = \alpha(k) \frac{1}{4} x^2 > 0
\end{equation*}
and 
\begin{equation*}
\begin{array}{c}
F(k,6) = \intt[6]{0} f(k,x) = \\\\
=\intt[2]{0} \alpha(k) \cdot \frac{1}{2} x + \intt[4]{2} \alpha(k) \cdot  \f{-x+3 \signum (x)} + \intt[6]{4} \alpha(k) \cdot \f{-\signum (x)} =\\\\
=  \alpha(k) + 0 - 2\alpha(k) < 0
\end{array}
\end{equation*}
Since function $g$ is negative in neighbourhood of $0$ for every $k\in [1,T]$ thus $G$ defined as $G(k,x) := \intt[x]{0} g(k,s) ds$ is nonpositive in this neighbourhood. On the other hand, for sufficiently large $\modul{t}$:
\begin{equation*}
G(k,t)\geq \beta(k) \frac{1}{4}\cdot \modul{t}.
\end{equation*}
Thus
\begin{equation*}
\liminf_{\modul{t} \to +\infty} \frac{G(k,t)}{\modul{t}} \geq \beta(k) \frac{1}{4	} > 0
\end{equation*}
which implies that \ref{A.4} holds.
\end{example}

\section{Remarks on a~three critical point theorems}


In this section we comment on some recently obtained
results pertaining to the existence of three critical points to action functionals.

\begin{theorem}\label{theo_from_Cabada}
Let $(X,\norm{\cdot} )$ be a~uniformly convex Banach space with strictly convex dual space, $J\in \Ck[1]{X}$ be a~functional with compact derivative, $x_{0},$ $x_{1}\in X$, $p,r \in \setR$ be such that $p>1$ and $r>0$. Let the following conditions be satisfied:

\begin{enumerate}[label=\textbf{(B.\arabic*)}]
\item $\liminf_{\norm{x}\to+\infty}\frac{J(x)}{\norm{x}^{p}}\geq 0$
\item $\inf_{x \in X} J(x)< \inf_{\norm{x-x_0}\leq r	} J(x) $
\item $\norm{x_1-x_0}<r$ and $J(x_1) < \inf_{\norm{x-x_0}=r} J(x)$.
\end{enumerate}
Then there exists a~nonempty open set $A\subseteq (0,+\infty )$ such that for all $\lambda \in A$ the functional $x\rightarrow \frac{\norm{x-x_{0}}^{p}}{p}+\lambda J(x)$ has at least three critical points in $X$.
\end{theorem}

The above theorem initiated some later research as concerning its applicability to anisotropic problems, see \cite{galewski.wieteska}, where the term $ \norm{x}^{p}$ is replaced by some convex coercive functional. Namely, the result from \cite{galewski.wieteska} reads:
\begin{theorem}\label{Galewski.Wieteska}
Let $(X,\norm{\cdot})$ be a~uniformly convex Banach space with strictly convex dual space, $J\in \Ck[1]{X,\setR}$ be a~functional with compact derivative, $\mu \in \Ck[1]{X,\setR_+}$ be a~convex coercive functional such that its derivative is an operator $\mu^{\prime}:X\rightarrow \dualspace{X}$ admitting a~continuous inverse, let $\widetilde{x}\in X$ and $r>0$ be fixed. Assume that the
following conditions are satisfied:
\begin{enumerate}[label=\textbf{(C.\arabic*)}]
\item $\liminf_{\norm{x}\to\infty} \frac{J(x)}{\mu(x)} \geq 0$ \label{C.1}
\item $\inf_{x \in X} J(x)< \inf_{\mu{x}\leq r} J(x) $ \label{C.2}
\item $\f[\mu]{\widetilde{x}}<r$ and $J(\widetilde{x}) < \inf_{\mu(x)=r} J(x)$. \label{C.3}
\end{enumerate}
Then there exists a~nonempty open set $A\subseteq (0,+\infty )$ such that for all $\lambda \in A$ the functional $\mu +\lambda J$ has at least three critical points in $X$.
\end{theorem}

\begin{remark}
Note that when $\mu(x) =\norm{x}^{p}$ then
Theorem \ref{theo_from_Cabada} follows from Theorem \ref{Galewski.Wieteska}.
\end{remark}

Some further question can be asked when examining assumptions and proof of Theorem \ref{Galewski.Wieteska}. Namely whether this is possible to weaken assumptions \ref{C.1}-\ref{C.3}. We try to answer these questions
in this note providing some related multiplicity result. In our
proof we will base on Theorem \ref{Galewski.Wieteska} and also on the following lemma, which can be easily
derived from \cite[Proposition 2.2]{ricceri.1} and \cite[Theorem 1]{bonnano}

\begin{lemma}\label{Key-lemma} 
Let $(X,\norm{\cdot})$ be a~reflexive Banach
space, $I\subseteq \setR_+$ be an interval, $\Phi \in \Ck[1]{X}$ be a~sequentially weakly l.s.c. functional whose derivative admits a~continuous inverse, $J\in \Ck[1]{X}$ be a~functional with compact derivative. Moreover, assume that there exist $x_{1},x_{2}\in X$ and $\sigma \in \setR$ such that:
\begin{enumerate}[label=\textbf{(D.\arabic*)}]
\item $\Phi (x_{1})<\sigma <\Phi (x_{2})$ \label{D.1}
\item $\inf_{\Phi (x)\leq \sigma} J(x)>\frac{\left(
\Phi \left( x_{2}\right) -\sigma \right) J(x_{1})+\left( \sigma -\Phi
(x_{1})\right) J(x_{2})}{\Phi (x_{2})-\Phi (x_{1})}$ \label{D.2}
\item $\lim_{\norm{x}\to \infty} \left[ \Phi (x)+\lambda J(x)\right] =+\infty $ for all $\lambda \in I$. \label{D.3}
\end{enumerate}
Then there exists a~nonempty open set $A\subseteq I$ such that for all $\lambda \in A$ the functional $\Phi +\lambda J$ has at least three critical points in X.
\end{lemma}

We will provide our main results in terms of a~kind of comparison theorems.
In this section we provide the following simple observation:

\begin{theorem}
Let $(X,\norm{\cdot})$ be a~uniformly convex Banach space with strictly convex dual space, $J\in \Ck[1]{X,\setR}$ be a~functional with compact derivative. $\mu _{1}\in \Ck[1]{X,\setR}$ and $\mu_{2}\in \Ck[1]{X,\setR_+}$ be a~convex coercive functional such that its derivative is an operator $\mu_{2}^{\prime }:X\rightarrow \dualspace{X}$ admitting a~continuous inverse, let $y\in X$ and $r>0$ be fixed. Assume the following conditions are
satisfied:
\begin{enumerate}[label=\textbf{(E.\arabic*)}]
\item $\liminf_{\norm{x}\to\infty} \frac{J(x)}{\mu_2(x)} \geq 0$ \label{E.1}
\item $\inf_{x \in X} J(x)< \inf_{\f[\mu_1]{x}\leq r} J(x) $ \label{E.2}
\item $\f[\mu_2]{\widetilde{x}}<r$ and $J(\widetilde{x}) < \inf_{\f[\mu_2]{x}=r} J(x)$. \label{E.3}
\item For all $x \in X$ if $\mu_2(x) \leq r$ then $\mu_1(x) \leq \mu_2(x)$.
\end{enumerate}
Then there exists a~non empty open set $A\subset (0,+\infty )$ such that for all $\lambda \in A$ the functional $x\rightarrow \mu_{2}(x)+\lambda J(x)$ has at least three critical points in $X$.
\end{theorem}

\begin{proof}
If $z\in \set{ x:\mu _{2}(x) \leq r} $ then $\mu_{1}(z)\leq r$. Thus 
\begin{equation*}
\inf_{\mu_{1}(x)\leq r}J(x) \leq \inf_{\mu_{2}\leq r} J(x)
\end{equation*}%
We apply Theorem \ref{Galewski.Wieteska} with $\mu :=\mu _{2}$.
\end{proof}

\section{A general multiplicity result}

In this section we provide our main result.

\begin{theorem}\label{3CP.Kowalski.Galewski}[Main Theorem]
Let $(X,\norm{\cdot})$ be a~uniformly convex Banach space with strictly convex dual space, $J\in \Ck[1]{X,\setR}$ be a~functional with compact derivative. Assume that $\mu_{1}\in \Ck[1]{X,\setR}$ is sequentially w.l.s.c and coercive. Let $\mu_{2}\in \Ck[1]{X,\setR_+}$ be a~convex coercive functional. Assume that derivative of $\mu_{1}$ is an operator $\mu_{1}^{\prime }:X\rightarrow \dualspace{X}$ admitting a~continuous inverse. Let $y\in X$ and $r>0$ be fixed. Assume the following conditions are satisfied:
\begin{enumerate}[label=\textbf{(F.\arabic*)}]
\item $\liminf_{\norm{x}\rightarrow \infty }\frac{J(x)}{\mu_{2}\left( x\right) }\geq 0$ \label{F.1}
\item $\inf_{x\in X}J(x)<\inf_{\mu _{1}\left( x\right) \leq r}J(x)$ \label{F.2}
\item $\mu_{2}\left( y\right) <r$ and $J(y)<\inf_{\mu_{2}\left( x\right) =r}J(x)$ \label{F.3}
\item $\foralll_{x\in X}\mu_{2}\left( x\right) \leq r \then \mu _{1}\left( x\right) \leq \mu _{2}\left( x\right) $ and $\mu_{1}\left( x\right) \geq \mu _{2}\left( x\right) $ for $\norm{x} \geq M$, where $M>0$ is some constant. \label{F.4}
\item $J$ is convex on the convex hull of $B:=\set{x\in X:\mu _{1}(x)\leq r}$ \label{F.5}
\end{enumerate}
Then there exists a~non empty open set $A\subset (0,+\infty )$ such that for all $\lambda \in A$ the functional $x \to \mu _{1}\left( x\right) +\lambda J(x)$ has at least three critical points.
\end{theorem}

\begin{proof}
We will use Lemma \ref{Key-lemma}. Set $I=(0,+\infty )$ and observe that for any $\lambda \in I$ we have for sufficiently large $\norm{x}$
by \ref{F.1} and \ref{F.4} that $\frac{J(x)}{\mu _{2}(x)}>-\frac{1}{2\lambda}$. Thus 
\begin{equation*}
\mu _{1}(x)+\lambda J(x)>\mu _{2}(x)-\lambda \frac{1}{2\lambda }\mu _{2}(x)= 
\frac{1}{2}\mu _{2}(x)\rightarrow +\infty
\end{equation*}
as $\norm{x}\to +\infty $. So we have condition\ref{D.3} of Lemma \ref{Key-lemma} satisfied.

We define $C:=\set{ x\in X:\mu _{2}(x)\leq r} $. We claim there
exists $x_{1}$ such that $\mu_{1}(x_1)<r$ and $J(x_{1})=\inf_{x\in B}J(x)$.
Note that $C\subset B$. Since $\mu _{2}$ is continuous and convex, the set $ C $ is weakly closed. Since $\mu _{2}$ is coercive, it follows that $C$ is weakly compact. Since $J$ has a~compact derivative, so it is s.w.l.s.c. and therefore its restriction to $C$ attains its infimum. We shall refer to its minimizer as $z$.

Take $y$ as in \ref{F.3}. We
can distinguish the three following cases
\begin{enumerate}[label=\textbf{Case \arabic*.}, ref=\textbf{case \arabic*}]
\item $y$ minimizes also $J$ over $B$. \label{main.proof.case.1}
\item $y$ does not minimize $J$ over $B$ but $z$ does. \label{main.proof.case.2} 
\item neither $y$ and nor $z$ minimize $J$ over $B$. \label{main.proof.case.3}
\end{enumerate}

In \ref{main.proof.case.1} we put $y=x_{1}$ since $r>\mu _{2}(y)\geq \mu _{1}(y)$.
Which proves the case.

In \ref{main.proof.case.2} we take $z=x_{1}$ since 
\begin{equation*}
J(z)=\inf_{x\in C}J(x)=\inf_{x\in B}J(x)
\end{equation*}
Suppose $z\in \boundary{C}$, then 
\begin{equation*}
J(z)=\inf_{x\in \boundary{C}}J(x)>J(y)>J(z)
\end{equation*}
contradiction. Thus $r>\mu _{2}(z)\geq \mu _{1}(z)$.

In \ref{main.proof.case.3} if neither $y$ and nor $z$ minimize $J$ in $B$, there would exist such $s\in B\setminus C$ such that $J(s)<J(z)\leq J(y)$. $C$ is convex and closed thus there would exists such $\alpha \in (0,1)$ that $t:=\alpha s+(1-\alpha )z\in \boundary{C}$. Then by \ref{F.5} we see that 
\begin{equation*}
J(t)\geq \inf_{x\in \boundary{C}}J(x)=J(z)>J(s)
\end{equation*}
Since $J$ in convex 
\begin{equation*}
J(t)\leq \alpha J(s)+(1-\alpha )J(z)<J(z)
\end{equation*}
We see that it is impossible. Thus we have $x_{1}$ such that $\mu _{1}(x_{1})<r$ and $J(x_{1})=\inf_{x\in B}J(x)$.

By \ref{F.2} there exist $x_{2}$ such that $\mu _{1}(x_{2})>r$ and $J(x_{2})<\inf_{x\in B}J(x)=J(x_{1})$.\bigskip\ Putting $\phi =\mu _{1}$, $\delta =r$ we see that condition \ref{D.1} of Lemma \ref{Key-lemma} is satisfied.

Finally 
\begin{equation*}
\begin{array}{l}
\inf_{x\in B}J(x)=J(x_{1})=\frac{J(x_{1})(\mu _{1}(x_{2})-\mu _{1}(x_{1}))}{
(\mu _{1}(x_{2})-\mu _{1}(x_{1}))}= \\ \\
=\frac{(\mu _{1}(x_{2})-r)J(x_{1})+(r-\mu _{1}(x_{1}))J(x_{1})}{(\mu
_{1}(x_{2})-\mu _{1}(x_{1}))}> \\ \\
>\frac{(\mu _{1}(x_{2})-r)J(x_{1})+(r-\mu _{1}(x_{1}))J(x_{2})}{(\mu
_{1}(x_{2})-\mu _{1}(x_{1}))}.
\end{array}
\end{equation*}
Thus condition \ref{D.2} of Lemma \ref{Key-lemma} holds.
\end{proof}

Since we aim at applications for finite dimensional systems and we will work
in a~finite dimensional Hilbert space, the assumptions of Theorem \ref{3CP.Kowalski.Galewski}
can be relaxed. Other types of discrete BVPs can also be considered with
this approach. In the finite dimensional context we obtain the following:

\begin{theorem} \label{3CP.HilbertSpace}Let $(X,\norm{\cdot})$ be a~finite dimensional
Hilbert space and let $J\in \Ck[1]{X,\setR}$ be a~functional with compact derivative. Assume that $\mu_{1}\in \Ck[1]{X,\setR}$ is coercive, and $\mu_{2}\in \Ck[1]{X,\setR_+}$ be a~convex coercive functional. Assume that there derivative of $\mu_{1}$ is an operator $\mu_{1}^{\prime }:X\rightarrow \dualspace{X}$ admitting a~continuous inverse, let $y\in X$ and $r>0$ be fixed. Assume the following conditions are satisfied:
\begin{enumerate}[label=\textbf{(G.\arabic*)}]
\item $\liminf_{\norm{x}\rightarrow \infty }\frac{J(x)}{\mu_{2}\left( x\right) }\geq 0$ \label{G.1}
\item $\inf_{x\in X}J(x)<\inf_{\mu _{1}\left( x\right) \leq r}J(x)$ \label{G.2}
\item $\mu_{2}\left( y\right) <r$ and $J(y)<\inf_{\mu_{2}\left( x\right) =r}J(x)$ \label{G.3}
\item $\foralll_{x\in X}\mu_{2}\left( x\right) \leq r \then \mu _{1}\left( x\right) \leq \mu _{2}\left( x\right) $ and $\mu_{1}\left( x\right) \geq \mu _{2}\left( x\right) $ for $\norm{x} \geq M$, where $M>0$ is some constant. \label{G.4}
\item $J$ is convex on the convex hull of $B:=\set{x\in X:\mu _{1}(x)\leq r}$ \label{G.5}
\end{enumerate}
Then there exists a~non empty open set $A\subset (0,+\infty )$ such that for all $\lambda \in A$ the functional $x \to \mu _{1}\left( x\right) +\lambda J(x)$ has at least three critical points.
\end{theorem}

\begin{proof}
We see that $X$ is a~Banach space with a~strictly convex dual. Since $X$ is
finite dimensional weak convergence is equivalent to the strong one, so $\mu 
$ and $J$ are weakly continuous. Thus we may apply Theorem \ref{3CP.Kowalski.Galewski}.
\end{proof}

\section{Existence and multiplicity results for problem \eqref{zad}}

\begin{lemma}\label{norm.equivalence}
For any $u \in X$ following inequality holds
\begin{equation*}
\frac{2}{\sqrt{T+1}}\norm[C]{u} \leq \norm{u} \leq 2 \sqrt{T} \norm[C]{u}
\end{equation*}
\end{lemma}
\begin{proof}
First it is obvious that
\begin{equation*}
2 \norm[C]{u} \leq \summ[T]{k=1} \modul{\deltaop[u]{k}}
\end{equation*}
By H\"olders inequality we know that
\begin{equation*}
\summ[T]{k=1} \modul{\deltaop[u]{k}} \leq \f{T+1} \sqrt{\summ[T+1]{k=1} \f{\deltaop[u]{k}}^2} = \f{T+1} \norm{u}
\end{equation*}
Which proves the first inequality. The other one is proven as follows:
\begin{equation*}
\begin{array}{c}
\norm{u} = \sqrt{\summ[T+1]{k=1} \f{\deltaop[u]{k}}^2} = 
\sqrt{\summ[T+1]{k=1} \f{\f[u]{k}-\f[u]{k-1}}^2} \leq \\\\
\sqrt{\summ[T]{k=1} 2 \cdot \norm[C]{u}^2 + 2 \cdot \norm[C]{u} \cdot \norm[C]{u} } = 2\sqrt{T} \norm[C]{u}. 
\end{array}
\end{equation*}
\end{proof}

\begin{proposition} \label{Theorem.Coerciveness}Let $p^{-}\geq 2$. Assume that conditions \ref{A.1}-\ref{A.5} hold. Then for all $\lambda >0$, $\gamma >0$ problem \eqref{zad} has at least one solution.
\end{proposition}
\begin{proof}
Let us define
\begin{equation*}
\mu_{1}(x)=\summ[T+1]{k=1}\f{ \frac{1}{p(k-1)} \modul{\deltaop[x]{k-1}}^{p(k-1)}+\gamma G(k,x(k)) },
\end{equation*}
\begin{equation*}
\mu _{2}(x)=\summ[T+1]{k=1}\f{ \frac{1}{p(k-1)} \modul{\deltaop[x]{k-1}}^{p(k-1)}}.
\end{equation*}
Let 
\begin{equation*}
J(x)=\summ[T]{k=1} F(k,x(k)).
\end{equation*}
Then $E_{\gamma ,\lambda }(x)=\mu _{1}(x)+\lambda J(x)$. 
Since in \cite{galewski.wieteska} it was shown that $\mu_{2}(x)\rightarrow \infty $ as $\norm{x} \to \infty $, so $\mu _{1}(x)\rightarrow \infty $ as $\norm{x} \to \infty $. Next we see that by \ref{A.1} it follows that $E_{\gamma ,\lambda }(x)\to \infty $ as $\norm{x} \to \infty $.
\bigskip

Since $E_{\gamma ,\lambda }$ is differentiable, continuous, coercive and $X$
is a~finite dimensional space, it has at least one critical point which is a
weak and thus a~strong solution to (\ref{zad}).
\end{proof}

As an application of Theorem \ref{3CP.HilbertSpace} to problem \eqref{zad} we have the
following:

\begin{theorem}\label{Main.Theorem} 
Let $p^{-}\geq 2$. Assume that conditions \ref{A.1}-\ref{A.5} hold. Then there exists $\gamma_{max} >0 $ such that for every $\gamma \in (0,\gamma_{max})$ there exists $A_\gamma \subseteq (0,+\infty )$ such that for
all $\lambda \in A$ problem \eqref{zad} has at least two nontrivial solutions.
\end{theorem}

\begin{proof}
We will show step by step that the assumptions of Lemma \ref{3CP.HilbertSpace} hold.
We will start by proving \ref{G.1}.

Let $\seq{x}{n}$ such that $\norm{x_n}\to \infty$. Let $\epsilon >0$. 
By norm equivalence there exists such $c>0$ that:
\begin{equation*}
\norm{x}^{p^{-}} \geq c \summ[T]{k=1} \modul{x(k)}^{p^{-}}.
\end{equation*}
We set $c_1 = \frac{c T^{\frac{2-p^{-}}{2}}}{2p^+}$. By \ref{A.1} there exists such $K_1 \in \setN$ that
\begin{equation*}
\foralll_{k \in [1,T]} \foralll_{\modul{t}>K_1} \frac{-F(k,t)}{\modul{t}^{p^{-}}} < \epsilon \frac{c_1}{T}
\end{equation*}
Let $K_2 \in \setN$ be such that for all $n>K_2$, $\norm{x_n}>1$. It is easy to see that
\begin{equation*}
\mu(x_n) \geq \frac{T^\frac{2-p^{-}}{2}}{p^{+}} \norm{x_n}^{p^{-}} - \frac{T+1}{p^{+}}
\end{equation*}
Then there exists $K_3 \geq \max \set{K_2,K_1}$ such that for all $n \geq K_3$
\begin{equation*}
\mu_2(x_n) \geq \frac{T^\frac{2-p^{-}}{2}}{2p^{+}} \norm{x_n}^{p^{-}} \geq c_1 \summ[T]{k=1} \modul{x_n(k)}^{p^{-}}
\end{equation*}
Let denote as $M= \max \set{\modul{F(k,t)} : k\in [1,T], \modul{t}\leq K_1}$. By coerciveness of $\mu_2$ there exists such $K_4$ that for all $n\geq K_4$
\begin{equation*}
\mu_2(x_n) \geq \frac{MT}{\epsilon}
\end{equation*}
Let $k = \max \set{K_3,K_4}$. Let $n \geq k$, then:
\begin{equation*}
\begin{array}{c}
-\frac{J(x_n)}{\mu_2(x_n)} = \frac{- \summ[T]{k=1} F(k,x_n(k))}{\mu_2(x_n)}  \leq \\
\leq \frac{ \summ[T]{k=1, \modul{x(k)}\leq K_1} \modul{F(k,x_n(k))}}{\mu_2(x_n)} + 
\frac{ \summ[T]{k=1, \modul{x(k)}> K_1} \max \{- F(k,x_n(k)) , 0\} }{\mu_2(x_n)} \leq \\
\leq \frac{ \summ[T]{k=1, \modul{x(k)}\leq K_1} \modul{F(k,x_n(k))}}{\mu_2(x_n)} + 
\frac{ \summ[T]{k=1, \modul{x(k)}> K_1} \max \{- F(k,x_n(k)) , 0\}}{c_1 \summ[T]{k=1} \modul{x_n(k)}^{p^{-}}} \leq \\
\leq \summ[T]{k=1, \modul{x(k)}\leq K_1} \frac{  M \epsilon }{M T} + 
\summ[T]{k=1, \modul{x(k)}> K_1} \frac{\epsilon}{T} = \epsilon
\end{array}
\end{equation*}
Thus 
\begin{equation*}
\foralll_{\epsilon >0} \existss_{k \in \setN} \foralll_{n \geq k} \quad \frac{J(x_n)}{\mu_2(x_n)	} \geq -\epsilon.
\end{equation*}
Which proves \ref{G.1}:
\begin{equation*}
\liminf_{\norm{x}\to + \infty} \frac{J(x)}{\mu_2(x)} \geq 0
\end{equation*}
Now we will prove \ref{G.4}. By coerciveness and $\f[\mu_2]{0}=0$ there exists $r^\ast >0$ such that
\begin{equation*}
\forall x \in X, \quad \mu_2(x) \leq r^\ast \then \norm{x} \leq 1.
\end{equation*}
Let $0<r < r_1 = \min \set{\f{\frac{2 M_1}{\sqrt{T+1}}}^{p^+} \frac{T^{\frac{p^{+}-2}{2}}}{p^+},r^\ast}$, and let $x\in X$ be such that $\mu_2(x) \leq r$. Then
\begin{equation*}
\mu_2(x) \leq \f{\frac{2 M_1}{\sqrt{T+1}}}^{p^+} \frac{T^{\frac{p^{+}-2}{2}}}{p^+}.
\end{equation*}
We know that
\begin{equation*}
\mu_2(x) = \summ[T+1]{k=1} \frac{1}{p(k-1)} \modul{\deltaop[x]{k-1}}^{p(k-1)} \geq \frac{1}{p^{+}} \summ[T+1]{k=1} \modul{\deltaop[x]{k-1}}^{p(k-1)}.
\end{equation*}
When $\norm{x}\leq 1$ it follows that
\begin{equation*}
\summ[T+1]{k=1} \modul{\deltaop[x]{k-1}}^{p(k-1)} \geq T^{\frac{p^{+}-2}{2}} \norm{x}^{p^{+}}.
\end{equation*}
So we have the following:
\begin{equation*}
\norm[C]{x}\leq  \frac{\sqrt{T+1}}{2} \norm{x}.
\end{equation*}
Thus 
\begin{equation*}
\norm{x}\leq \frac{2 M_1}{\sqrt{T+1}} \text{ and } \norm[C]{x}\leq M_1.
\end{equation*}
By \ref{A.4} for all $k \in [1,T]$ we see that $G(k,x(k)) \leq 0$ and so $\summ[T]{k=1}G(k,x(k)) \leq 0$. For any $\gamma >0$ this implies that $\mu_1(x) \leq \mu_2(x)$. We will prove the second part of condition \ref{G.4}. From \ref{A.4} we have that
\begin{equation*}
\existss_{d>0} \foralll_{k \in [1,T]} \quad \liminf_{\modul{t}\to \infty} \frac{G(k,t)}{\modul{t}} > d.
\end{equation*}
Let $t_k >0 $ be such real number that
\begin{equation*}
\foralll_{k \in [1,T]} \foralll_{\modul{t}>t_k} \quad \frac{G(k,t)}{\modul{t}} > \frac{d}{2}.
\end{equation*}
By $G$ continuity it is obvious that there exists such $l<0$ that
\begin{equation*}
\foralll_{k \in [1,T]} \foralll_{t \in\setR} \quad G(k,t) \geq l.
\end{equation*}
Let $x \in X$ such that 
\begin{equation*}
\norm{x} \geq M := \min \set{2 \sqrt{T} t_k, \frac{- 4 l T\sqrt{T}}{d}} >0
\end{equation*}
By 
\begin{equation*}
\norm{x}\leq 2 \sqrt{T} \norm[C]{x}
\end{equation*}
we conclude that $\norm[C]{x} \geq \max \set{t_k, \frac{-2lT}{d}}$. Let $q \in [1,T]$ be such index that $\norm[C]{x} = \modul{x(q)}$. Then
\begin{equation*}
\summ[T+1]{k=1} G(k,x(k)) = \summ[T+1]{q\neq k=1} G(k,x(k)) + G(q,x(q)) \geq l\cdot T + G(q,x(q)).
\end{equation*}
Since $\modul{x(q)}=\norm[C]{x}>t_k$ then $G(q,x(q)) > \frac{d}{2} \modul{x(q)}$. Moreover, since $\modul{x(q)}=\norm[C]{x}> \frac{-2lT}{d}$ we obtain that
\begin{equation*}
l \cdot T + G(q,x(q)) \geq l \cdot T + \frac{d}{2} \modul{x(q)} \geq l \cdot T + \frac{d}{2} \cdot \frac{-2lT}{d} = 0.
\end{equation*}
then for every $\gamma >0$ we conclude that $\mu_1(x) \geq \mu_2(x)$ which proves the case \ref{G.4}. We will now prove \ref{G.2}. Let 
\begin{equation*}
r < r_2 = \min \set{\f{\frac{2 M_1}{\sqrt{T+1}}}^{p^+} \frac{T^{\frac{p^{+}-2}{2}}}{p^+},\f{\frac{2 m}{\sqrt{T+1}}}^{p^+} \frac{T^{\frac{p^{+}-2}{2}}}{p^+},r^\ast}
\end{equation*}
For such $r<r_2\leq r_1$ the proof of \ref{G.4} holds. Lets define $\gamma_{max} = \frac{r_2 - r}{-(T+1) l}$. Let $x \in X$ such that $\mu_1(x) \leq r$, and $\gamma \in (0,\gamma_{max})$. We observe that
\begin{equation*}
\mu_1(x) = \mu_2(x) + \gamma \summ[T+1]{k=1} G(k,x(k)) \leq r.
\end{equation*}
We have the following chain of estimations:
\begin{equation*}
\begin{array}{c}
\mu_2(x) \leq r - \gamma \summ[T+1]{k=1} G(k,x(k))\leq r - \gamma l (T+1) \leq r - \gamma_{max} l (T+1)\\\\
\leq r - \frac{r_2-r}{-(T+1) l } l (T+1) \leq r_2
\end{array}
\end{equation*}
Since $r_2 \leq \f{\frac{2 m}{\sqrt{T+1}}}^{p^+} \frac{T^{\frac{p^{+}-2}{2}}}{p^+}$ we obtain, in similar way as before
\begin{equation*}
\mu_2(x) \leq \f{\frac{2 m}{\sqrt{T+1}}}^{p^+} \frac{T^{\frac{p^{+}-2}{2}}}{p^+}.
\end{equation*}
Since
\begin{equation*}
\summ[T+1]{k=1} \modul{\deltaop[x]{k-1}}^{p(k-1)} \leq \f{\frac{2 m}{\sqrt{T+1}}}^{p^+} T^{\frac{p^{+}-2}{2}}
\end{equation*}
we see that:
\begin{equation*} 
\norm{x} \leq \frac{2m}{\sqrt{T+1}} \text{ and }\norm[C]{x}\leq m
\end{equation*}
Which proves that $J(x) \geq 0$. Since $x$ was taken arbitrary we have that $\inf_{\mu_1(x)\leq r} J(x) \geq 0$. On the other hand, if we choose 
\begin{equation*}
x(k) = \left \lbrace
\begin{array}{c l}
\frac{s_1+s_2}{2} &, k \in [1,T] \\
0 &, k=0 \lor k=T+1
\end{array}
 \right.
\end{equation*}
then $J(x) = \summ[T]{k=2} F(k,\frac{s_1+s_2}{2}) < 0$. Which proves \ref{G.2}
\begin{equation*}
\inf_{x \in X} J(x) < 0 \leq \inf_{\mu_1(x) \leq r} J(x).
\end{equation*}
By \ref{A.5} $J$ is convex on $\set{x: \norm[C]{x}\leq m}$. Thus it is convex on convex hull of $\set{x: \mu_1(x) \leq r}$ since it is the smallest convex set that contains $\set{x: \mu_1(x) \leq r}$. Then \ref{G.5} holds. Finally we prove \ref{G.3}. For the same $r < r_2$, let $x \in X$ such that $\mu_2(x) \leq r$. Then $\mu_2(x) \leq r_2$. Since $r_2 \leq \f{\frac{2 m}{\sqrt{T+1}}}^{p^+} \frac{T^{\frac{p^{+}-2}{2}}}{p^+}$ we obtain, in similar way as before
\begin{equation*}
\mu_2(x) \leq \f{\frac{2 m}{\sqrt{T+1}}}^{p^+} \frac{T^{\frac{p^{+}-2}{2}}}{p^+}.
\end{equation*}
Since
\begin{equation*}
\summ[T+1]{k=1} \modul{\deltaop[x]{k-1}}^{p(k-1)} \leq \f{\frac{2 m}{\sqrt{T+1}}}^{p^+} T^{\frac{p^{+}-2}{2}}
\end{equation*}
we see that:
\begin{equation*} 
\norm{x} \leq \frac{2m}{\sqrt{T+1}} \text{ and }\norm[C]{x}\leq m
\end{equation*}
Thus $J(x) \geq 0$. Let $y=0$. Then off course  $J(y)=0$. Let $z \in X$ such that $\mu_2(z)=r$. We know that $J(z) \geq 0$ and we will prove in fact $J(z) >0$. Indeed, it is obvious that
\begin{equation*}
\mu_2(x)=0 \iff x=0
\end{equation*} 
Since $0<r=\mu_2(z) \then z \neq 0$. Since $z \neq 0$ then there exists such $q \in [1,T]$ that $0<\norm[C]{z}=\modul{z(q)}$. Then 
\begin{equation} 
J(z) \geq \f[F]{q,z(q)} > 0.
\end{equation}
Finally
\begin{equation}
\inf_{\mu_2(x) = r} J(x) = \min_{\mu_2(x) = r} J(x) >0 = J(0)=J(y)
\end{equation}
which completes the proves by proving \ref{G.2}.
\end{proof}

\begin{tabular}{l}
Marek Galewski, Piotr Kowalski \\ 
Institute of Mathematics, \\ 
Technical University of Lodz, \\ 
Wolczanska 215, 90-924 Lodz, Poland, \\ 
marek.galewski@p.lodz.pl, piotr.kowalski.1@p.lodz.pl%
\end{tabular}

\end{document}